\documentclass{svproc}
\usepackage{url}

\usepackage{hyperref}
\usepackage{amssymb}
\usepackage{amsfonts}
\usepackage{tikz,tikz-cd,float}
\usepackage{latexsym}
\usepackage{mathtools}
\usepackage{mathrsfs}
\usepackage{graphicx}
\usepackage{bbm}
\usepackage{color}
\usepackage{bm}
\usepackage{todonotes}
\usepackage{enumerate}

\usepackage{caption}
\usepackage{subcaption}

\numberwithin{equation}{section}

\newcommand{\bx}{\mathbf{x}}
\newcommand{\by}{\mathbf{y}}

\newcommand{\bu}{\mathbf{u}}

\newcommand{\dd}{\,\mathrm{d}}

\newcommand{\dt}{\, \mathrm{d}t}

\begin{document}
\mainmatter              % start of a contribution
\title{{ Blow-up} of strong solutions of the  Thermal Quasi-Geostrophic equation}
%
%\titlerunning{Hamiltonian Mechanics}  % abbreviated title (for running head)
%                                     also used for the TOC unless
%                                     \toctitle is used
%
\author{Dan Crisan, Prince Romeo Mensah}
%
%\authorrunning{Dan Crisan, Prince Romeo Mensah} % abbreviated author list (for running head)
%
%%%% list of authors for the TOC (use if author list has to be modified)
%\tocauthor{Ivar Ekeland, Roger Temam, Jeffrey Dean, David Grove,
%Craig Chambers, Kim B. Bruce, and Elisa Bertino}
%
\institute{Department of Mathematics, Imperial College, London SW7 2AZ, United Kingdom\\
\email{d.crisan@imperial.ac.uk, p.mensah@imperial.ac.uk},
\\ WWW home page:
\texttt{https://www.ma.ic.ac.uk/~dcrisan/}
\\ WWW home page:
\texttt{https://www.imperial.ac.uk/people/p.mensah}
}

\maketitle              % typeset the title of the contribution

\begin{abstract}
The Thermal Quasi-Geostrophic  { (TQG) }equation is a coupled system of equations that governs the evolution of the buoyancy and the potential vorticity of a fluid. It has a local in time solution as proved in \cite{crisan2021theoretical}. In this paper, we give a criterion for the blow-up of solutions to the  Thermal Quasi-Geostrophic equation, in the spirit of the classical Beale--Kato--Majda blow-up criterion 
(cf. \cite{beale1984remarks}) for the solution of the Euler equation.

%The abstract should summarize the contents of the paper
%using at least 70 and at most 150 words. It will be set in 9-point
%font size and be inset 1.0 cm from the right and left margins.
%There will be two blank lines before and after the Abstract. \dots

% We would like to encourage you to list your keywords within
% the abstract section using the \keywords{...} command.
\end{abstract}
\section{Introduction}

The  Thermal Quasi-Geostrophic (TQG) equation is a coupled system of equations governed by the evolution of the buoyancy { $b:(t,\bx)\in [0,T] \times \mathbb{R}^2 \mapsto b(t,\bx)\in \mathbb{R}$} and the potential vorticity { $q:(t,\bx)\in [0,T] \times \mathbb{R}^2 \mapsto q(t,\bx)\in \mathbb{R}$} in the following way:
\begin{align}
\partial_t b + (\bu \cdot \nabla) b =0,
\label{ce}
\\
\partial_t q + (\bu\cdot \nabla)( q -b) = -(\bu_h \cdot \nabla) b,  \label{me}
%\\
%{ \dd q + (\bu\cdot \nabla)q\dt + (\bm{\xi}_i\cdot\nabla) q \circ \dd W^i= (\bu-\bu_h) \cdot \nabla b \dt} \label{me2}
\\
b(0,x)=b_0(x),\qquad q(0,x)=q_0(x),
\label{initialTQG}
\end{align}
where
\begin{align}
\label{constrt}
\bu =\nabla^\perp \psi, 
\qquad
\bu_h =\frac{1}{2} \nabla^\perp h,
\qquad
q=(\Delta -1) \psi +f.
\end{align}
Here, { $\psi:(t,\bx)\in [0,T] \times \mathbb{R}^2 \mapsto \psi(t,\bx)\in \mathbb{R}$} is the streamfunction, { $h:\bx\in   \mathbb{R}^2 \mapsto h(\bx)\in \mathbb{R}$} is the spatial variation around a constant bathymetry profile and { $f:\bx\in   \mathbb{R}^2 \mapsto f(\bx)\in \mathbb{R}$} is the Coriolis parameter. Since we are working on the whole space, we can supplement our system with the far-field condition
\begin{align*}
\lim_{\vert \bx\vert \rightarrow \infty}
(b(\bx), \bu(\bx)) =0.
\end{align*}
Our given set of data is  $( \bu_h, f, b_0, q_0)$ with regularity class:
\begin{equation}
\begin{aligned}
\label{dataMain}
\bu_h \in W^{3,2}_{\mathrm{div}}(\mathbb{R}^2;\mathbb{R}^2), \quad  f\in W^{2,2}(\mathbb{R}^2), \quad
b_0\in  W^{3,2}(\mathbb{R}^2) , \quad q_0 \in  W^{2,2}(\mathbb{R}^2).
\end{aligned}
\end{equation}
The TQG equation models the dynamics of a submesoscale geophysical fluid in thermal geostrophic balance, for which the Rossby number, the Froude number and the stratification parameter are all of the same asymptotic order.
{  For a historical overview, modelling and other issues pertaining to the TQG equation, we refer the reader to \cite{crisan2021theoretical}.} 
\\
In the following, we are interested in \textit{strong solutions} of the system \eqref{ce}--\eqref{constrt}
{ 
which can naturally be defined  in terms of just $b$ and $q$ although the unknowns in the evolutionary equations \eqref{ce}--\eqref{me} are $b$, $q$ and $\bu$. This is because for a given $f$, one can recover the  velocity $\bu$ from the vorticity $q$ by solving the equation
\begin{align*}
\bu =\nabla^\perp(\Delta-1)^{-1}(
q-f)
\end{align*}
derived from \eqref{constrt}. Also note that a consequence of the equation $\bu =\nabla^\perp \psi$ in \eqref{constrt} is that $\mathrm{div} \bu=0$. This means that the fluid is incompressible.
}
With these information in hand, we now make precise,  the notion of a strong solution.
\begin{definition}[Local strong solution]
\label{def:solution}
Let $(\bu_h,f,b_0, q_0)$ be of regularity class \eqref{dataMain}. For some $T>0$, we call the triple $(b,  q,T) $
a   \textit{strong solution} to the system \eqref{ce}--\eqref{constrt} if the following holds:
\begin{itemize}
\item The buoyancy $b$ satisfies $b \in  C([0,T]; W^{3,2}(\mathbb{R}^2))$
and the equation
\begin{align*}
b(t) &= b_0 -  \int_0^{t} \mathrm{div} (b\bu)\,\dd \tau, 
\end{align*}
holds for all $t\in[0,T]$;
\item  the potential vorticity $ q$ satisfies $ q \in  C([0,T]; W^{2,2}(\mathbb{R}^2))$ 
and the equation
\begin{align*}
 q (t)  &=  q_0 - \int_0^t \Big[ \mathrm{div}  (( q-b)\mathbf{u} ) 
+
\mathrm{div} (b\bu_h ) \Big] \,\dd \tau
\end{align*}
holds for all $t\in[0,T]$.
\end{itemize}
\end{definition}
Such local strong solutions exist on a maximal time interval. We define this as follows.
\begin{definition}[Maximal solution]\label{def:MaxSolution}
{ 
Let $(\bu_h,f,b_0, q_0)$ be of regularity class \eqref{dataMain}. For some $T>0$, we} call  $(b, q,  T_{\max}) $
a  \textit{maximal solution} to the system \eqref{ce}--\eqref{constrt} if:
\begin{itemize}
\item there exists an increasing sequence of time steps $(T_n)_{n\in \mathbb{N}}$ whose limit is $T_{\max}\in (0,\infty]$;
\item  for each $n\in \mathbb{N}$, the triple $(b, q, T_n) $ is
a local strong solution to the system \eqref{ce}--\eqref{constrt} with initial condition $(b_0, q_0) $;
\item if $T_{\max}<\infty$, then
\begin{align}
\label{limsupSol}
\limsup_{T_n\rightarrow T_{\max}} \Vert b(T_n) \Vert_{W^{3,2}(\mathbb{R}^2)}^2+\Vert q(T_n) \Vert_{W^{2,2}(\mathbb{R}^2)}^2 =\infty.
\end{align} 
\end{itemize}
We shall call $T_{\max}>0$ the \textit{maximal time}.
\end{definition}
The existence of a unique {  local} strong solution of \eqref{ce}--\eqref{constrt} has recently been shown in \cite[Theorem 2.10]{crisan2021theoretical} on the torus. A unique maximal solutions also exist \cite[Theorem 2.14]{crisan2021theoretical} and the result also applies to the whole space \cite[Remark 2.1]{crisan2021theoretical}. We state the result here for completeness.
\begin{theorem}
\label{thm:MaxSol}
For $(\bu_h,f,b_0, q_0)$ of regularity class \eqref{dataMain}, there exist a unique maximal solution $(b,  q, T) $ of the system \eqref{ce}--\eqref{constrt}.
\end{theorem}
Before we state our main result, let us first present some notations used throughout this work.
\subsection{Notations}
In the following, we write $F \lesssim G$  if there exists  a generic constant $c>0$ { (that may vary from line to line)}  such that $F \leq c\,G$.
{  Functions mapping into $\mathbb{R}^2$ are \textbf{boldfaced} (for example the velocity $\bu$) while those mapping into $\mathbb{R}$ are not (for example the buoyancy $b$ and vorticity $q$).
For $k\in \mathbb{N}\cup\{0\}$ and $p\in [1,\infty]$,  $W^{k,p}(\mathbb{R}^2)$ is the usual Sobolev space of functions mapping into  $\mathbb{R}$ with a natural modification for functions mapping into $\mathbb{R}^2$.} For $p=2$, $W^{k,2}(\mathbb{R}^2)$ is a Hilbert space  with  inner product
$
\langle u,v \rangle_{W^{k,2}(\mathbb{R}^2)} =\sum_{\vert \beta\vert\leq k} \langle \partial^\beta u\,,\, \partial^\beta v \rangle,
$
where $\langle\cdot\,,\,\rangle$ denotes the standard $L^2$-inner product.  
For general  $s\in\mathbb{R}$, we use the norm
\begin{align}
\label{sobolevNorm}
\Vert  v\Vert_{W^{s,2}(\mathbb{R}^2)}
\equiv
 \bigg(\int_{\mathbb{R}^2} \big(1+\vert \xi\vert^2  \big)^s\vert  \widehat{v}(\xi)\vert^2
  \dd \xi\bigg)^\frac{1}{2}
\end{align}
defined in frequency space. Here, $\widehat{v}(\xi)$ denotes the Fourier coefficients  of $v$.
For simplicity, we  write $\Vert  \cdot\Vert_{s,2}$ for $\Vert \cdot\Vert_{W^{s,2}(\mathbb{R}^2)}$.
When $k=s=0$, we get the usual $L^2(\mathbb{R}^2)$ space whose norm we will simply denote by $\Vert \cdot \Vert_2$.
A similar notation will be used for norms $\Vert \cdot \Vert_p$ of general $L^p(\mathbb{R}^2)$ spaces for any $p\in [1,\infty]$ as well as for the inner product $\langle\cdot,\cdot \rangle_{k,2}:=\langle\cdot,\cdot \rangle_{W^{k,2}(\mathbb{R}^2)}$ when $k\in \mathbb{N}$. Additionally, $W^{k,p}_{\mathrm{div}}(\mathbb{R}^2)$ represents the space of divergence-free vector-valued functions in $W^{k,p}(\mathbb{R}^2)$.
\\
With respect to differential operators, we let 
 $\nabla_0 :=(\partial_{x_1},\partial_{x_2}, 0)^T$ and $\nabla_0^\perp :=(-\partial_{x_2},\partial_{x_1}, 0)$ be the three-dimensional extensions of the  two-dimensional differential operators $\nabla=(\partial_{x_1},\partial_{x_2})^T$ and $\nabla^\perp :=(-\partial_{x_2},\partial_{x_1})$  by zero respectively. The Laplacian $\Delta =\mathrm{div}\nabla= \partial_{x_1x_1}+\partial_{x_2x_2}$ remains two-dimensional.
\subsection{Main result}
Our main result is to give a blow-up criterion, of Beale--Kato--Majda-type \cite{beale1984remarks}, for a strong solution $(b,  q, T) $ of \eqref{ce}--\eqref{constrt}.
In particular, we show the following result.

\begin{theorem}
\label{thm:BKM3}
Suppose that $(b, q,T)$ is a local strong solution of \eqref{ce}--\eqref{constrt}. If
\begin{align}
\label{xmnot3}
\int_0^T \Big(\Vert q (t) \Vert_{\infty}+ \Vert \nabla b (t) \Vert_{\infty}\Big) \dt \equiv K<\infty,
\end{align}
then there exists a solution $(b',  q', T')$ with $T'>T$, such that $(b',q')=(b,q)$ on $[0,T]$.
Moreover, for all $t\in[0,T]$, 
\begin{equation*}
\Vert b(t) \Vert_{3,2} + \Vert q(t) \Vert_{2,2}
\leq
\big[
\mathrm{e}+ \Vert b_0 \Vert_{3,2} + \Vert q_0 \Vert_{2,2}\big]^{\exp(cK)}
\exp[cT\exp(cK)].
\end{equation*}

\end{theorem}
An immediate consequence of the above theorem is the following:
\begin{corollary}
Assume that $(b,q,T)$ is a maximal solution.
%, i.e., the solution pair $(b,  q)$ cannot be extended to $[0,T')$ for any $T'>T$. 
If $T<\infty$, then 
\begin{align*}
\int_0^T \Big(\Vert q (t) \Vert_{\infty}+ \Vert \nabla b (t) \Vert_{\infty}\Big) \dt = \infty
\end{align*}
and in particular,
\begin{align*}
\sup_{t\uparrow T}\Big(\Vert q (t) \Vert_{\infty}+ \Vert \nabla b (t) \Vert_{\infty}\Big) = \infty.
\end{align*}
\end{corollary}
 
\section{Blow-up}
We devote the entirety of this section to the proof of Theorem \ref{thm:BKM3}. In order to achieve our goal, we first derive a suitable exact solution for what is referred to as the modified Helmholtz equation. Some authors also call it the Screened Poisson equation \cite{bhat2008fourier} while others rather  mistakenly call it the Helmholtz equation. {  Refer to \cite{arfken2005mathematical} for the difference between the  Helmholtz equation and modified Helmholtz equation.}
\subsection{Estimate for the $2D$ modified Helmholtz equation or the screened Poisson equation}

In the following, we want to find an exact solution $\psi: \mathbb{R}^2\rightarrow \mathbb{R}$ of
\begin{align}
\label{0000}
(\Delta - 1)\psi(\mathbf{x}) = w(\mathbf{x}),  \qquad \lim_{\vert \mathbf{x}\vert \rightarrow \infty}\psi(\mathbf{x})=0
\end{align}
for a given function $w \in W^{2,2}( \mathbb{R}^2)$
%that is sufficiently smooth
. The corresponding two-dimensional free space Green's function $G^\text{free}(\mathbf{x})$ for \eqref{0000} must therefore solve
\begin{align}
\label{000}
(\Delta - 1)G^\text{free}(\mathbf{x}- \mathbf{y}) = \delta(\mathbf{x}- \mathbf{y}),  \qquad \lim_{\vert \mathbf{x}\vert \rightarrow \infty} G^\text{free}(\mathbf{x}- \mathbf{y})(\mathbf{x})=0
\end{align}
in the sense of distributions. 
Indeed, one can verify that the Green's function is given by 
\begin{align}
G^\text{free}(\mathbf{x}- \mathbf{y})= \frac{1}{2\pi} K_0(\vert\mathbf{x}- \mathbf{y}\vert) 
%=\frac{i}{4} H^{(1)}_0(i\vert\mathbf{x}- \mathbf{y}\vert),
\end{align}
see \cite[Table 9.5]{arfken2005mathematical}, where
\begin{align*}
K_0(z)
= \int_0^\infty
\frac{e^{-\sqrt{z^2+r^2}}}{\sqrt{z^2+r^2}} \, \mathrm{d} r
\end{align*}
is the modified Bessel function of the second kind, see
equation (8.432-9), page 917 of
\cite{gradshteyn2007table} with $\nu=0$ and $x=1$. 
%Note that { the \textit{Gamma function}} $\Gamma(1/2)=\sqrt{\pi}$.
However, since the integral above is an even function, it follows that
\begin{align}
G^\text{free}(\mathbf{x}- \mathbf{y})
=\frac{i}{4} H^{(1)}_0(i\vert\mathbf{x}- \mathbf{y}\vert)
=\frac{1}{4\pi} \int_\mathbb{R}
\frac{e^{-\sqrt{\vert\mathbf{x}- \mathbf{y}\vert^2+r^2}}}{\sqrt{\vert\mathbf{x}- \mathbf{y}\vert^2+r^2}} \, \mathrm{d} r
\end{align}
which is  the zeroth-order Hankel function of the first kind, see 
equation (11.117) in \cite{arfken2005mathematical}  and equation (8.421-9) of \cite{gradshteyn2007table} on page 915. Therefore,
\begin{align}
\label{psiExplict}
\psi(\mathbf{x})
%&=\int_{\mathbb{R}^2}G^\text{free}(\mathbf{x}- \mathbf{y})w( \mathbf{y}) \mathrm{d}\mathbf{y} 
%\\&=
%\frac{1}{4\pi}\int_{\mathbb{R}^2}\int_{\mathbb{R}}
%\frac{e^{-\sqrt{\vert \mathbf{x}- \mathbf{y}\vert^2 +r^2}}}{\sqrt{\vert \mathbf{x}- \mathbf{y}\vert^2 +r^2}}w( \mathbf{y})  \, \mathrm{d} r
%\mathrm{d}\mathbf{y} 
%\\
&=
\frac{1}{4\pi}\int_{\mathbb{R}^3}
\frac{e^{-\vert (\mathbf{x}-\mathbf{y},-r) \vert}}{\vert (\mathbf{x}-\mathbf{y},-r) \vert} w( (\mathbf{y},0))  \, \mathrm{d}\mathbf{y}\mathrm{d} r  
\\&=:\psi((\mathbf{x},0))
\end{align}
where we have used the identity $\sqrt{\vert \mathbf{x}- \mathbf{y}\vert^2 +r^2}= \vert (\mathbf{x},0)-(\mathbf{y},r)\vert = \vert (\mathbf{x}-\mathbf{y},-r) \vert$.
We can therefore view the argument of the streamfunction $\psi$ as a $3D$-vector with  zero vertical component.
\subsection{Log-Sobolev estimate for velocity gradient}
Our goal now is to find a suitable estimate for the Lipschitz norm of $\bu$ that solves
 \begin{align}
 \label{screenedPoison}
 \bu = \nabla^\perp \psi, \qquad (\Delta-1)\psi= w
 %, \qquad w=q-f
 \end{align}
 where $w\in W^{2,2}(\mathbb{R}^2)$ is given. In particular, inspired by  \cite{beale1984remarks}, we aim to show {  Proposition \ref{prop:gradUEst} below. This log-estimate is the crucial ingredient that allow us to obtain our blow-up criterion in terms of just the buoyancy gradient and the vorticity although preliminary estimate may have suggested estimating the velocity gradient as well.}
 \begin{proposition}
 \label{prop:gradUEst}
 For a given $w\in W^{2,2}(\mathbb{R}^2)$, any $\bu$ solving \eqref{screenedPoison} satisfies
 \begin{equation}
\begin{aligned}
\label{velovortiEst}
\Vert  \mathbf{u} \Vert_{1,\infty}
\lesssim
1
+
(1+2\ln^+(\Vert w \Vert_{2,2}))
\Vert w 
\Vert_{\infty}
\end{aligned}
\end{equation}
where $\ln^+ a = \ln\, a$ if $a\geq1$ and $\ln^+a=0$ otherwise.
\end{proposition}
\begin{proof}
To show \eqref{velovortiEst}, we fix $L\in (0,1]$ and for $\mathbf{z}\in \mathbb{R}^3$, we let $\zeta_L(\mathbf{z})$ be a smooth cut-off function satisfying
\begin{align*}
\zeta_L(\mathbf{z})
=
\begin{cases}
1 & : \vert \mathbf{z} \vert<L,\\
0 & : \vert \mathbf{z} \vert>2L
\end{cases}
\end{align*}
and $ \vert\partial \zeta_L(\mathbf{z}) \vert  \lesssim L^{-1}$ where $\partial:=\nabla^\perp_0$ or $\nabla_0
$ as well as  $ \vert\nabla_0 \nabla^\perp_0 \zeta_L(\mathbf{z}) \vert  \lesssim L^{-2}$. This latter requirement ensures that the point of inflection of the graph of the cut-off, the portion that is constant, concave upwards and concave downwards are all captured. 
We now define the following
\begin{align*}
&B_1:= \big\{(\mathbf{y},r) \in \mathbb{R}^3 \, : \, \vert (\mathbf{x},0) - (\mathbf{y},r) \vert 
=
\vert (\mathbf{x}-\mathbf{y},-r) \vert
<2L  \big\},
\\
&B_2:= \big\{(\mathbf{y},r) \in \mathbb{R}^3 \, : \, L \leq 
\vert (\mathbf{x}-\mathbf{y},-r) \vert
\leq1  \big\},
\\
&B_3:= \big\{(\mathbf{y},r) \in \mathbb{R}^3 \, : \, 
\vert (\mathbf{x}-\mathbf{y},-r) \vert
>1  \big\},
\end{align*}
so that {  by adding and subtracting $\zeta_L$, we obtain} 
\begin{align*}
\vert\nabla \mathbf{u}(\mathbf{x})
\vert
=
\vert
\nabla_0 \nabla^\perp_0
\psi((\mathbf{x},0))
\vert
&\leq
\vert
\nabla_0 (\mathbf{u}_1((\mathbf{x},0)),0)
\vert
+
\vert
\nabla_0 (\mathbf{u}_2^1((\mathbf{x},0)),0)
\vert
\\&+
\vert
\nabla_0 (\mathbf{u}_2^2((\mathbf{x},0)),0)
\vert
+
\vert
\nabla_0 (\mathbf{u}_2^3((\mathbf{x},0)),0)
\vert
\\&+
\vert
\nabla_0 (\mathbf{u}_2^4((\mathbf{x},0)),0)
\vert
+
\vert
\nabla_0 (\mathbf{u}_3((\mathbf{x},0)),0)
\vert
\\&
=:
\vert
\nabla_0 \mathbf{u}_1
\vert
+
\vert
\nabla_0 \mathbf{u}_2^1
\vert
+
\vert
\nabla_0 \mathbf{u}_2^2
\vert
+
\vert
\nabla_0 \mathbf{u}_2^3
\vert
\\&
+
\vert
\nabla_0 \mathbf{u}_2^4
\vert
+
\vert
\nabla_0 \mathbf{u}_3
\vert
\end{align*}
where 
\begin{align*}
\nabla_0 \mathbf{u}_1
&:=
\frac{1}{4\pi}\int_{B_1}
\zeta_L((\mathbf{x}-\mathbf{y},-r))
\frac{e^{-\vert (\mathbf{x}-\mathbf{y},-r)\vert}}{\vert (\mathbf{x}-\mathbf{y},-r)\vert} \nabla_0 \nabla^\perp_0 w((\mathbf{y},0))  \,
\mathrm{d}\mathbf{y} \mathrm{d} r,
\\
\nabla_0 \mathbf{u}_2^1
&:=
\frac{1}{4\pi}\int_{B_2}
\big[1-
\zeta_L((\mathbf{x}-\mathbf{y},-r))\big]
\nabla_0 \nabla^\perp_0
\bigg[
\frac{e^{-\vert (\mathbf{x}-\mathbf{y},-r)\vert}}{\vert (\mathbf{x}-\mathbf{y},-r)\vert}  \bigg]w((\mathbf{y},0))  \, 
\mathrm{d}\mathbf{y}\mathrm{d} r,
\\
\nabla_0 \mathbf{u}_2^2
&:=
\frac{1}{4\pi}\int_{B_2}
\nabla_0 \nabla^\perp_0
\big[1-
\zeta_L((\mathbf{x}-\mathbf{y},-r))\big]
\frac{e^{-\vert (\mathbf{x}-\mathbf{y},-r)\vert}}{\vert (\mathbf{x}-\mathbf{y},-r)\vert} \, w( (\mathbf{y},0))  \,
\mathrm{d}\mathbf{y}  \mathrm{d} r,
\\
\nabla_0 \mathbf{u}_2^3
&:=
\frac{1}{4\pi}\int_{B_2}
 \nabla^\perp_0
\big[1-
\zeta_L((\mathbf{x}-\mathbf{y},-r))\big]
\nabla_0\bigg[\frac{e^{-\vert (\mathbf{x}-\mathbf{y},-r)\vert}}{\vert (\mathbf{x}-\mathbf{y},-r)\vert} \bigg]\, w( (\mathbf{y},0))  \,
\mathrm{d}\mathbf{y}  \mathrm{d} r,
\\
\nabla_0 \mathbf{u}_2^4
&:=
\frac{1}{4\pi}\int_{B_2}
 \nabla_0
\big[1-
\zeta_L((\mathbf{x}-\mathbf{y},-r))\big]
\nabla^\perp_0\bigg[\frac{e^{-\vert (\mathbf{x}-\mathbf{y},-r)\vert}}{\vert (\mathbf{x}-\mathbf{y},-r)\vert} \bigg]\, w( (\mathbf{y},0))  \,
\mathrm{d}\mathbf{y}  \mathrm{d} r,
\\
\nabla_0 \mathbf{u}_3
&:=
\frac{1}{4\pi}\int_{B_3}
\nabla_0 \nabla^\perp_0
\bigg[ \big[1-
\zeta_L((\mathbf{x}-\mathbf{y},-r)) \big]
\frac{e^{-\vert (\mathbf{x}-\mathbf{y},-r)\vert}}{\vert (\mathbf{x}-\mathbf{y},-r)\vert} \bigg]w((\mathbf{y},0))  \, 
\mathrm{d}\mathbf{y} \mathrm{d} r.
\end{align*}
For $L\in(0,1]$, we have that
\begin{align*}
\vert \nabla_0 \mathbf{u}_1 \vert
&\lesssim
\bigg(\int_{B_1}
\frac{e^{-2\vert (\mathbf{x},0)-(\mathbf{y},r)\vert}}{\vert (\mathbf{x},0)-(\mathbf{y},r)\vert^2} 
  \, \mathrm{d}\mathbf{y}\mathrm{d} r 
\bigg)^\frac{1}{2}
\Vert\nabla_0 \nabla^\perp_0 w( (\mathbf{y},0))
\Vert_2
\\
&\lesssim
\bigg(\int_0^{2L}
\frac{e^{-2s}}{s^2} 
  \, s^2\mathrm{d} s 
\bigg)^\frac{1}{2}
\Vert w 
\Vert_{2,2}
\lesssim
\big(1-
e^{-4L} 
\big)^\frac{1}{2}
\Vert w 
\Vert_{2,2}
\lesssim
\pi 
L^\frac{1}{2}
\Vert w 
\Vert_{2,2}.
\end{align*}
Now note that
\begin{align*}
\nabla_0 \mathbf{u}_2^1
&:=
\frac{1}{4\pi}\int_{B_2}
 \big[1-
\zeta_L((\mathbf{x}-\mathbf{y},-r)) \big]
\bigg\{
\frac{2( \mathbf{x}- \mathbf{y})^T( \mathbf{x}- \mathbf{y})^\perp }{(\vert \mathbf{x}- \mathbf{y}\vert^2 +r^2)^2}
+
\frac{3( \mathbf{x}- \mathbf{y})^T( \mathbf{x}- \mathbf{y})^\perp }{(\vert \mathbf{x}- \mathbf{y}\vert^2 +r^2)^\frac{5}{2}}
\\
&-
\frac{1 }{\vert \mathbf{x}- \mathbf{y}\vert^2 +r^2}
 \begin{pmatrix}
  0 & 1 & 0\\
  -1 & 0 & 0\\
  0 & 0 & 0
 \end{pmatrix}
-
\frac{1 }{(\vert \mathbf{x}- \mathbf{y}\vert^2 +r^2)^\frac{3}{2}}
 \begin{pmatrix}
  0 & 1 & 0 \\
  -1 & 0 & 0 \\
  0 & 0 & 0
 \end{pmatrix}
\\
&+
\frac{( \mathbf{x}- \mathbf{y})^T( \mathbf{x}- \mathbf{y})^\perp  }{(\vert \mathbf{x}- \mathbf{y}\vert^2 +r^2)^\frac{3}{2}}
+
\frac{( \mathbf{x}- \mathbf{y})^T( \mathbf{x}- \mathbf{y})^\perp  }{(\vert \mathbf{x}- \mathbf{y}\vert^2 +r^2)^2}
 \bigg\}
e^{-\vert (\mathbf{x}-\mathbf{y},-r)\vert} w( \mathbf{y})  \, \mathrm{d} r
\mathrm{d}\mathbf{y} 
\\&
=:
\sum_{i=1}^6
\mathbb{K}_i((\mathbf{x}-\mathbf{y},-r)).
\end{align*}
Clearly,
$
\vert \mathbf{x}-\mathbf{y} \vert^2\leq \vert \mathbf{x}-\mathbf{y} \vert^2+r^2=\vert(\mathbf{x}-\mathbf{y},-r)\vert^2$ and for any $L\in(0,1]$, the inequalities
\begin{align*}
{ 
(e^{-L} - e^{-1})
\leq (
1 - e^{-1})
\leq (1 - e^{-1})(1-\ln(L))
\lesssim (1-\ln(L))
}
\end{align*}
holds independent of $L$. Therefore, for $L\in(0,1]$, it follows that
\begin{align*}
\vert \mathbb{K}_1((\mathbf{x}-\mathbf{y},-r)) \vert 
&+
\vert \mathbb{K}_3((\mathbf{x}-\mathbf{y},-r))
+
\vert \mathbb{K}_6((\mathbf{x}-\mathbf{y},-r)) \vert
\\&\lesssim \Vert w \Vert_\infty
\int_{B_2}
\frac{
e^{-\vert(\mathbf{x}-\mathbf{y},-r)\vert}}{\vert(\mathbf{x}-\mathbf{y},-r)\vert^2}  \, \mathrm{d} r
\mathrm{d}\mathbf{y} 
\\&
\lesssim \Vert w \Vert_\infty
\int_L^1
\frac{
e^{-s}}{s^2} s^2 \, \mathrm{d}s
\\&
\lesssim \Vert w \Vert_\infty
(1-\ln(L)).
\end{align*}
Again, we can use $
\vert \mathbf{x}-\mathbf{y} \vert^2\leq \vert \mathbf{x}-\mathbf{y} \vert^2+r^2$ and the fact that the inequalities
\begin{align*}
{ 
(e^{-L}(L+1) - 2e^{-1})
\leq  (
1 - 2e^{-1})
\leq (1 - 2e^{-1})(1-\ln(L))
\lesssim (1-\ln(L))
}
\end{align*}
holds independent of any $L\in(0,1]$ to obtain
\begin{align*}
\vert \mathbb{K}_5((\mathbf{x}-\mathbf{y},-r)) \vert 
&\lesssim \Vert w \Vert_\infty
\int_{B_2}
\frac{
e^{-\vert(\mathbf{x}-\mathbf{y},-r)\vert}}{\vert(\mathbf{x}-\mathbf{y},-r)\vert}  \, \mathrm{d} r
\mathrm{d}\mathbf{y} 
\\
&\lesssim \Vert w \Vert_\infty
\int_L^1
\frac{e^{-s}}{s}s^2\, \mathrm{d} s 
%\\&
%\lesssim \Vert w \Vert_\infty
%(
%e^{-L}(L+1) - 2e^{-1})
\\&
\lesssim \Vert w \Vert_\infty
(1-\ln(L)).
\end{align*}
Finally, for $\mathbb{K}_2$ and $\mathbb{K}_4$, we also obtain
%\begin{align*}
%\vert \mathbb{K}_2((\mathbf{x}-\mathbf{y},-r)) \vert 
%+
%\vert \mathbb{K}_4((\mathbf{x}-\mathbf{y},-r))
%&\lesssim \Vert h \Vert_\infty
%\int_{B_2}
%\frac{
%e^{-\vert(\mathbf{x}-\mathbf{y},-r)\vert}}{\vert(\mathbf{x}-\mathbf{y},-r)\vert^3}  \, \mathrm{d} r
%\mathrm{d}\mathbf{y} 
%\\
%&\lesssim \Vert h \Vert_\infty
%\frac{1}{L}
%\int_{B_2}
%\frac{
%e^{-\vert(\mathbf{x}-\mathbf{y},-r)\vert}}{\vert(\mathbf{x}-\mathbf{y},-r)\vert^2}  \, \mathrm{d} r
%\mathrm{d}\mathbf{y} 
%\\&
%\lesssim \Vert h \Vert_\infty
%\frac{1}{L}
%(
%e^{-L} - e^{-1})
%\end{align*}
%Or
%\begin{align*}
%\vert \mathbb{K}_2((\mathbf{x}-\mathbf{y},-r)) \vert 
%+
%\vert \mathbb{K}_4((\mathbf{x}-\mathbf{y},-r))
%&\lesssim \Vert h \Vert_\infty
%\int_{B_2}
%\frac{
%e^{-\vert(\mathbf{x}-\mathbf{y},-r)\vert}}{\vert(\mathbf{x}-\mathbf{y},-r)\vert^3}  \, \mathrm{d} r
%\mathrm{d}\mathbf{y} 
%\\
%&\lesssim \Vert h \Vert_\infty
%\int_L^1
%\frac{
%e^{-s}}{s^3} s^2 \, \mathrm{d}s
%\\&
%\lesssim \Vert h \Vert_\infty
%(
%\mathrm{Ei}[-1] - \mathrm{Ei}[-L]).
%\end{align*}
%Or
\begin{align*}
\vert \mathbb{K}_2((\mathbf{x}-\mathbf{y},-r)) \vert 
+
\vert \mathbb{K}_4((\mathbf{x}-\mathbf{y},-r))
&\lesssim \Vert w \Vert_\infty
\int_{B_2}
\frac{
e^{-\vert(\mathbf{x}-\mathbf{y},-r)\vert}}{\vert(\mathbf{x}-\mathbf{y},-r)\vert^3}  \, \mathrm{d} r
\mathrm{d}\mathbf{y} 
\\
&\lesssim \Vert w \Vert_\infty
\int_L^1
\frac{
e^{-s}}{s^3} s^2 \, \mathrm{d}s
\\&
\lesssim \Vert w \Vert_\infty
e^{-L}\int_L^1
\frac{
1}{s} \, \mathrm{d}s
\\&
\lesssim \Vert w \Vert_\infty
e^{-0}\big(-\ln(L)\big).
\end{align*}
We have shown that
\begin{equation}
\begin{aligned}
\vert \nabla_0 \mathbf{u}_2^1  \vert
\lesssim \Vert w \Vert_\infty
(1-\ln(L))
\end{aligned}
\end{equation}
for $L\in(0,1]$. 
Also, the quantity $(1/L^2)[e^{-L}(L+1)-e^{-2L}(2L+1) ]
$ is uniformly bounded for any $L\in(0,1]$ and as such,
\begin{equation}
\begin{aligned}
\vert \nabla_0 \mathbf{u}_2^2 \vert
&\lesssim
\bigg(\int_L^{2L}
\frac{e^{-s}}{sL^2} 
  \, s^2\mathrm{d} s 
\bigg)
\Vert w
\Vert_{\infty}
\lesssim
\Vert w
\Vert_{\infty}.
\end{aligned}
\end{equation}
{ 
Next, we note that the estimate for $ \nabla_0 \mathbf{u}_2^3$ and  $\nabla_0 \mathbf{u}_2^4$ will be the same where in particular,
\begin{align*}
\nabla_0 \mathbf{u}_2^3
&:=
\frac{-1}{4\pi}\int_{B_2}
\nabla_0^\perp \big[1-
\zeta_L((\mathbf{x}-\mathbf{y},-r)) \big]
\bigg\{
\frac{( \mathbf{x}- \mathbf{y})^T }{\vert \mathbf{x}- \mathbf{y}\vert^2 +r^2 }
\\&+
\frac{( \mathbf{x}- \mathbf{y})^T }{(\vert \mathbf{x}- \mathbf{y}\vert^2 +r^2)^\frac{3}{2}}
 \bigg\}
e^{-\vert (\mathbf{x}-\mathbf{y},-r)\vert} w( \mathbf{y})  \, \mathrm{d} r
\mathrm{d}\mathbf{y} 
\\&
=:
\mathbb{K}_7((\mathbf{x}-\mathbf{y},-r)) + \mathbb{K}_8((\mathbf{x}-\mathbf{y},-r)).
\end{align*}
Since $\vert \bx-\by\vert\leq 1$ holds on $B_2$, it follows from the condition $ \vert\nabla_0^\perp \zeta_L(\mathbf{z}) \vert  \lesssim L^{-1}$ that
\begin{align*}
\vert \mathbb{K}_7((\mathbf{x}-\mathbf{y},-r)) \vert 
&\lesssim\bigg(
\int_L^{2L}
\frac{
e^{-s}}{s^2L} s^2  \, \mathrm{d} s
\bigg)  \Vert w \Vert_\infty 
\lesssim \Vert w \Vert_\infty
\end{align*}
since $(1/L)[e^{-L}-e^{-2L}]$ is uniformly bounded in $L$. Similarly, we can use the fact that $\vert \bx-\by\vert\leq \sqrt{\vert \mathbf{x}- \mathbf{y}\vert^2 +r^2 }$ to obtain
\begin{align*}
\vert \mathbb{K}_8((\mathbf{x}-\mathbf{y},-r)) \vert 
&\lesssim\bigg(
\int_L^{2L}
\frac{
e^{-s}}{s^2L} s^2  \, \mathrm{d} s
\bigg)  \Vert w \Vert_\infty
\lesssim \Vert w \Vert_\infty.
\end{align*}
We can therefore conclude  that,
\begin{equation}
\begin{aligned}
\vert \nabla_0 \mathbf{u}_2^3  \vert + \vert \nabla_0 \mathbf{u}_2^4  \vert
\lesssim  \Vert w \Vert_\infty.
\end{aligned}
\end{equation}
}
Similar to the estimate for $\nabla \mathbf{u}_2^1$, we have that
\begin{equation}
\begin{aligned}
\vert \nabla_0 \mathbf{u}_3  \vert
\lesssim  \Vert w \Vert_\infty.
\end{aligned}
\end{equation}
It follows by summing up the various estimates above that
\begin{equation}
\begin{aligned}
\label{nabXinfty}
\Vert \nabla \bu \Vert_{\infty}
\lesssim
L^\frac{1}{2}\Vert w \Vert_{2,2}
+
(1-\ln(L))
\Vert w 
\Vert_{\infty}.
\end{aligned}
\end{equation}
It remains to show that the estimate \eqref{nabXinfty} also holds for $\bu$. 
{ 
For this, we first recall that
\begin{equation}
\begin{aligned}
\bu(\bx) &=
\frac{1}{4\pi} \int_{\mathbb{R}^3 }  
 \nabla^\perp_0
\bigg[
\frac{e^{-\vert (\mathbf{x}-\mathbf{y},-r)\vert}}{\vert (\mathbf{x}-\mathbf{y},-r)\vert}  \bigg]w((\mathbf{y},0))  \, 
\mathrm{d}\mathbf{y}\mathrm{d} r.
\end{aligned}
\end{equation}
We now use the inequalities
\begin{align}
\vert \bx-\by \vert \leq \big(\vert \bx-\by \vert^2 + r^2)^\frac{1}{2}
\end{align}
and
\begin{align*}
\frac{1}{4\pi}
\bigg\vert\nabla^\perp_0
\frac{e^{-\vert (\mathbf{x}-\mathbf{y},-r)\vert}}{\vert (\mathbf{x}-\mathbf{y},-r)\vert}  \bigg\vert
\lesssim
\bigg[
\frac{\vert\mathbf{x}-\mathbf{y}\vert}{ (\vert\mathbf{x}-\mathbf{y}\vert^2+r^2)^\frac{3}{2}} +
\frac{\vert \mathbf{x}-\mathbf{y} \vert}{\vert\mathbf{x}-\mathbf{y}\vert^2+r^2}   \bigg]e^{-\vert (\mathbf{x}-\mathbf{y},-r)\vert}
\end{align*}
to obtain
\begin{equation}
\begin{aligned}
\label{uinfty}
\Vert \bu \Vert_\infty &\lesssim
\Vert w \Vert_\infty
\int_{\mathbb{R}^3} \bigg[
\frac{1}{ \vert\mathbf{x}-\mathbf{y}\vert^2+r^2 }
+
\frac{1}{ (\vert\mathbf{x}-\mathbf{y}\vert^2+r^2)^\frac{1}{2}}   \bigg]e^{-\vert (\mathbf{x}-\mathbf{y},-r)\vert}
\mathrm{d}\mathbf{y}\mathrm{d} r
\\&
\lesssim
\Vert w \Vert_\infty
\int_0^{\infty}
\frac{e^{-s}}{s^2}
s^2
\mathrm{d}s
+
\Vert w \Vert_\infty
\int_0^{\infty}
\frac{e^{-s}}{s}
s^2
\mathrm{d}s
\\&
\lesssim
\Vert w \Vert_\infty.
\end{aligned}
\end{equation}
Therefore, it follows from \eqref{nabXinfty} and  \eqref{uinfty} that
}
\begin{equation}
\begin{aligned}
\Vert \bu \Vert_{1,\infty}
\lesssim
L^\frac{1}{2}\Vert w \Vert_{2,2}
+
(1-\ln(L))
\Vert w 
\Vert_{\infty}.
\end{aligned}
\end{equation}
If $\Vert w \Vert_{2,2} \leq 1$, we choose $L=1$ and if $\Vert w \Vert_{2,2} > 1$, we take $L=\Vert w \Vert_{2,2}^{-2}$ so that \eqref{velovortiEst} holds. This finishes the proof.
\end{proof}
Before we end the subsection, we also note that a direct computation using the definition of Sobolev norms in frequency space \eqref{sobolevNorm} immediately yield
\begin{align}
\label{masterEst}
\Vert \bu \Vert_{k+1,2} \lesssim \Vert w \Vert_{k,2}
\end{align}
for any $k\in \mathbb{N}\cup \{0\}$ where $w\in W^{k,2}(\mathbb{R}^2)$ is a given function in \eqref{screenedPoison}.

\subsection{A priori estimate}

In order to prove Theorem \ref{thm:BKM3}, 
we first need some preliminary estimates for $(b,q)$. In the following, we define
\begin{align*}
\Vert (b,q) \Vert:= \Vert b\Vert_{3,2}+\Vert q \Vert_{2,2}.
\end{align*}
\begin{lemma}
\label{lem:apriori}
A strong solution of \eqref{ce}--\eqref{constrt} satisfies the bound
\begin{equation}
\begin{aligned}
 \frac{\dd}{\dd t} \Vert  (b,q)  \Vert^2 
&\lesssim
\big( 1+\Vert \bu  \Vert_{1,\infty}
+
\Vert \nabla b  \Vert_{\infty} + \Vert q  \Vert_\infty \big)\big(1+ \Vert  (b,q)  \Vert^2 \big).
\end{aligned}
\end{equation}
\end{lemma}
\begin{proof}
Since the space of smooth functions is dense in the space $W^{3,2}(\mathbb{R}^2)\times W^{2,2}(\mathbb{R}^2)$ of existence, in the following, we work with a smooth solution pair $(b,q)$. 
To achieve our desired estimate, we apply $\partial^\beta $ to \eqref{ce} for $\vert \beta\vert \leq3$ to obtain
\begin{align}
\partial_t\partial^\beta b  + \bu  \cdot \nabla\partial^\beta  b   = R_1
\label{ce1}
\end{align} 
where 
\begin{align*}
R_1:= \bu  \cdot \partial^\beta  \nabla b 
-
\partial^\beta(\bu  \cdot \nabla  b  ).
\end{align*}
Now since $\mathrm{div}\bu  =0$, if we multiply \eqref{ce1} by $\partial^\beta b $
{ 
and integrate over space, the second term on the left-hand side of \eqref{ce1} vanishes after integration by parts. On the other hand, we can use the commutator estimate (see for instant \cite[Section 2.2]{crisan2021theoretical}) to estimate the residual term $R_1$.
Consequently, by multiplying \eqref{ce1} by $\partial^\beta b $, integrating over space, and
} summing over the {  multiindices} $\beta$ so that $\vert \beta\vert \leq3$,  we obtain 
\begin{equation}
\begin{aligned}
\label{xb32est}
\frac{\dd}{\dd t}
\Vert  b  \Vert_{3,2}^2 
&\lesssim
\Big(
\Vert \nabla \bu  \Vert_{\infty}\Vert b  \Vert_{3,2}
+
\Vert \nabla b  \Vert_{\infty} \Vert \bu  \Vert_{3,2}
\Big) 
\Vert  b  \Vert_{3,2}
\\&
\lesssim
(\Vert \nabla \bu  \Vert_{\infty}
+
\Vert \nabla b  \Vert_{\infty})(1+\Vert  (b,q)  \Vert^2)
\end{aligned}
\end{equation} 
where we have used \eqref{masterEst} for $w=q-f$ and $k=2$.
\\
Next, we find a bound for $\Vert  q  \Vert^2_{2,2}$. For this, we apply $\partial^\beta $ to \eqref{me} for $\vert \beta \vert \leq2$ and we obtain
\begin{align}
\partial_t\partial^\beta q  + \bu \cdot \nabla\partial^\beta(  q -b ) +\bu_h \cdot \nabla\partial^\beta b  =  R_2 + R_3+R_4
\label{xme1}
\end{align} 
where 
\begin{align*}
R_2&:= \bu \cdot \partial^\beta\nabla  q 
-
\partial^\beta(\bu  \cdot \nabla   q  ),
\\
R_3&:= -\bu \cdot \partial^\beta\nabla b 
+
\partial^\beta(\bu  \cdot \nabla  b  ),
\\
R_4&:= \bu_h\cdot \partial^\beta\nabla b 
-
\partial^\beta(\bu_h \cdot \nabla  b  ).
\end{align*}
Now notice that for $\mathbb{U}:=\nabla \bu $, it follows from interpolation that
\begin{align*}
\Vert \nabla \mathbb{U}\Vert_4\lesssim \Vert \mathbb{U} \Vert_\infty^\frac{1}{2}  \Vert \nabla^2 \mathbb{U}\Vert_2^\frac{1}{2} 
\end{align*}
and so,
\begin{align*}
\Vert \nabla^2 \bu \Vert_4\lesssim \Vert \nabla \bu \Vert_\infty^\frac{1}{2} \Vert \bu\Vert_{3,2}^\frac{1}{2}. 
\end{align*}
Similarly
\begin{align*}
\Vert \nabla q \Vert_4\lesssim \Vert q \Vert_\infty^\frac{1}{2} \Vert q\Vert_{2,2}^\frac{1}{2}. 
\end{align*}
Therefore,
\begin{align*}
\Vert \nabla q \Vert_4 \Vert \nabla^2 \bu \Vert_4
\lesssim \Vert q \Vert_\infty \Vert \bu\Vert_{3,2}
+
 \Vert \nabla\bu\Vert_\infty \Vert q\Vert_{2,2}.
\end{align*}
By using this estimate, we deduce from \eqref{masterEst} and commutator estimates that
\begin{align}
\Vert R_2 \Vert_2 &\lesssim 
\Vert \nabla \bu  \Vert_{\infty}\Vert  q  \Vert_{2,2}
+\Vert  q  \Vert_\infty (1+\Vert  q  \Vert_{2,2}). \label{xestR2}
\end{align}
The commutators $R_3$ and $R_4$ are easy to estimate and are given by
\begin{align}
\Vert R_3 \Vert_2 &\lesssim 
\Vert \nabla \bu  \Vert_{\infty} \Vert b  \Vert_{3,2}
+
\Vert \nabla b  \Vert_{\infty} (1+\Vert  q  \Vert_{2,2})
,\label{xestR3}
\\
\Vert R_4 \Vert_2&\lesssim 
\Vert b  \Vert_{3,2}
+
\Vert \nabla b  \Vert_{\infty}, \label{xestR4}
\end{align}
respectively, for a given $\bu_h \in W^{3,2}(\mathbb{R}^2;\mathbb{R}^2)$.
Next, by using  $\mathrm{div}\bu  =0$, we obtain
\begin{align}
 \big\langle(\bu \cdot \nabla\partial^\beta  q )\, ,\,\partial^\beta  q  \big\rangle
 %=\frac{1}{2}\int_{\mathbb{R}^2}\mathrm{div}(\bu \vert\partial^\beta  q \vert^2)\dd x
 =0.
\end{align}
Additionally, the following estimates  holds true
\begin{align}
\Big\vert \big\langle ( \bu \cdot \nabla\partial^\beta b  ) \, ,\,\partial^\beta q  \big\rangle \Big\vert
&\lesssim
\Vert \bu  \Vert_\infty \Vert b \Vert_{3,2}^2
+
\Vert \bu  \Vert_\infty 
\Vert  q \Vert_{2,2}^2,
\\
\Big\vert \big\langle ( \bu_h\cdot \nabla\partial^\beta b  ) \, ,\,\partial^\beta  q  \big\rangle \Big\vert
&\lesssim
\Vert b \Vert_{3,2}^2 
+ \Vert  q \Vert_{2,2}^2
\end{align}
since $\bu_h \in W^{3,2}(\mathbb{R}^2;\mathbb{R}^2)$. 
If we now collect the estimates above (keeping in mind that $f\in W^{2,2}(\mathbb{R}^2)$ and $\bu_h \in W^{3,2}(\mathbb{R}^2;\mathbb{R}^2)$), we obtain by multiplying \eqref{me} by $\partial^\beta  q $  and then summing over $\vert \beta\vert \leq2$, the following
\begin{equation} 
\begin{aligned}
\label{xomeg22}
 \frac{\dd}{\dd t} \Vert  q  \Vert_{2,2}^2
&\lesssim
\big( 1+\Vert \bu  \Vert_{1,\infty}
+
\Vert \nabla b  \Vert_{\infty} + \Vert q  \Vert_\infty \big)\big(1+\Vert  (b,q)  \Vert^2\big).
\end{aligned}
\end{equation}
Summing up \eqref{xb32est} and \eqref{xomeg22} yields the desired result.
\end{proof}
We now have all in hand to { prove} our main theorem, Theorem \ref{thm:BKM3}.

\begin{proof}[Proof of Theorem \ref{thm:BKM3}]
In the following, we define the time-dependent function $g$ as
\begin{align}
g(t):= \mathrm{e}+\Vert (b,q)(t) \Vert, \quad\text{for} \quad t\in[0,T].
\end{align}
Next, without loss of generality, we assume that $f=0$ so that from Proposition \ref{prop:gradUEst}, we obtain
\begin{equation}
\begin{aligned}
\Vert  \mathbf{u}(t) \Vert_{1,\infty}
\lesssim
1
+
(1+\ln\Vert q(t) \Vert_{2,2})
( \Vert \nabla b(t) 
\Vert_{\infty}
+
\Vert q (t)
\Vert_{\infty})
\end{aligned}
\end{equation}
for $t\in[0,T]$.
Using the monotonic properties of logarithms, it follows from the above that
\begin{equation}
\begin{aligned}
\Vert  \mathbf{u}(t) \Vert_{1,\infty}
\lesssim
1
+
\ln[g(t)]
( \Vert \nabla b (t)
\Vert_{\infty}
+
\Vert q (t)
\Vert_{\infty}).
\end{aligned}
\end{equation}
Furthermore, since $1\leq \ln(\mathrm{e}+|x|)$ for any $x\in \mathbb{R}$, we can deduce from the inequality above that
\begin{equation}
\begin{aligned}
\label{allGradEst}
\Vert  \mathbf{u}(t) \Vert_{1,\infty}
+
\Vert \nabla b(t) \Vert_\infty
+
\Vert q(t) \Vert_\infty
\lesssim
1
+
\ln[g(t)]
( \Vert \nabla b (t)
\Vert_{\infty}
+
\Vert q (t)
\Vert_{\infty}).
\end{aligned}
\end{equation}
On the other hand, it follows from Lemma \ref{lem:apriori} that
\begin{equation}
\begin{aligned}
\label{gEst}
g(t)
\leq
g(0)
\exp\bigg(c\int_0^t
\big( 1+\Vert \bu (s) \Vert_{1,\infty}
+
\Vert \nabla b (s) \Vert_{\infty} + \Vert q(s)  \Vert_\infty \big)\dd s \bigg)
\end{aligned}
\end{equation}
for any $t\in[0,T]$.
Combining \eqref{allGradEst} and \eqref{gEst} yields
\begin{equation}
\begin{aligned}
g(t)
\leq
g(0)
\exp\bigg(c\int_0^t
\big( 1
+
\ln[g(s)]
( \Vert \nabla b (s)
\Vert_{\infty}
+
\Vert q (s)
\Vert_{\infty}) \big)\dd s \bigg).
\end{aligned}
\end{equation}
We can now take logarithm of both sides and apply Gr\"onwall's lemma to the resulting inequality to obtain
\begin{equation}
\begin{aligned}
\label{lastbutone}
\ln[g(t)]
\leq
\big(\ln[
g(0)]+cT\big)
\exp\bigg(c\int_0^t
( \Vert \nabla b (s)
\Vert_{\infty}
+
\Vert q (s)
\Vert_{\infty})\dd s \bigg).
\end{aligned}
\end{equation}
At this, point, we can now utilize \eqref{xmnot3},  take exponentials in \eqref{lastbutone} and obtain
\begin{equation}
\begin{aligned}
\Vert (b,q)(t) \Vert
\leq
[
g(0)]^{\exp(cK)}
\exp[cT\exp(cK)]
\end{aligned}
\end{equation}
for any $t\in[0,T]$. Since the right-hand side is finite, it follows that the solution $(b,q)$ can be continued on some interval  $[0,T')$ for some $T'>T$
% by taking $(b,q)(T)$ as an initial condition, say
 . 
 This finishes the proof. 
\end{proof}

\section*{Acknowledgements}
This work has been supported by the European Research Council (ERC) Synergy grant STUOD-DLV-856408.

%
% ---- Bibliography ----
%

\end{document}